\documentclass[oneside, a4paper]{amsart}
	\usepackage[a4paper, margin=0.9in]{geometry}
	\usepackage{amsmath, amsthm, amssymb, latexsym, mathtools, mathrsfs, eucal}
	\usepackage{dsfont, soul, stackrel, tikz-cd, graphicx, etoolbox, multirow}
	\DeclareGraphicsExtensions{.pdf,.png,.jpg}
	\usepackage{fancyhdr}
	\usepackage[shortlabels]{enumitem}
	\usepackage[pdfmenubar=true, pdfborder={0 0 0 [3 3]}, colorlinks, citecolor=magenta]{hyperref}
	\numberwithin{equation}{section}
	\setenumerate{listparindent=\parindent}

	\newtheoremstyle{Mytheorem}%
	{1em}{1em}%
	{\slshape}{}%
	{\bfseries}{.}%
	{ }{}
 
	\newtheoremstyle{Mydefinition}%
	{1em}{1em}%
	{}{}%
	{\bfseries}{.}%
	{ }{}

	\theoremstyle{Mydefinition}
	\newtheorem{statement}{Statement}[section]
	\newtheorem{definition}[statement]{Definition}
	
	\newtheorem{remark}[statement]{Remark}
	
	\newtheorem{example}[statement]{Example}

	\newtheorem*{comment*}{Comment}

	\theoremstyle{Mytheorem}
	\newtheorem{theorem}[statement]{Theorem}
	\newtheorem{corollary}[statement]{Corollary}
	
	\newtheorem{proposition}[statement]{Proposition}
	\newtheorem{lemma}[statement]{Lemma}

	\newcommand{\G}{GL_2^{+}(\mathbb{Q})}

	\newcommand{\pa}[1]{\frac{\partial}{\partial{#1}}}
	\newcommand{\prt}[2]{\frac{\partial{#1}}{\partial{#2}}}

	\newcommand{\nc}{\newcommand}
	\newcommand{\be}{\begin{eqnarray*}}
	\newcommand{\ee}{\end{eqnarray*}}
	\newcommand{\bea}{\begin{eqnarray}}
	\newcommand{\eea}{\end{eqnarray}}
	\newcommand{\bs}{\begin{split}}
	\newcommand{\es}{\end{split}}
	\newcommand{\bal}{\begin{align}}
	\newcommand{\eal}{\end{align}}

	\nc{\bei}{\begin{itemize}}
	\nc{\eei}{\end{itemize}}
	\nc{\bee}{\begin{enumerate}}
	\nc{\eee}{\end{enumerate}}
	\nc{\bet}{\begin{thm}}
	\nc{\eet}{\end{thm}}
	\nc{\bed}{\begin{defn}}
	\nc{\eed}{\end{defn}}
	\nc{\bel}{\begin{lem}}
	\nc{\eel}{\end{lem}}
	\nc{\bep}{\begin{prop}}
	\nc{\eep}{\end{prop}}
	\nc{\bec}{\begin{corollary}}
	\nc{\eec}{\end{corollary}}
	\nc{\ber}{\begin{rem}}
	\nc{\eer}{\end{rem}}
	\nc{\beex}{\begin{example}}
	\nc{\eeex}{\end{example}}
	\nc{\bpm}{\begin{pmatrix}}
	\nc{\epm}{\end{pmatrix}}
	\nc{\bspm}{\left(\begin{smallmatrix}}
	\nc{\espm}{\end{smallmatrix}\right)}

	\newcommand{\cA}{\mathcal{A}}
	\newcommand{\cB}{\mathcal{B}}
	\newcommand{\cC}{\mathcal{C}}

	\newcommand{\cO}{\mathcal{O}}

	\newcommand{\cT}{\mathcal{T}}

	\newcommand{\bZ}{\mathbb{Z}}

	\nc{\frf}{\mathfrak{f}}

	\nc{\frs}{\mathfrak{s}}  
	\nc{\frt}{\mathfrak{t}} 
	\nc{\fru}{\mathfrak{u}}
	\nc{\lsl}{\mathfrak{sl}}
	\nc{\lgl}{\mathfrak{gl}}
	\nc{\upsi}{\underline{\psi}}
	\nc{\uchi}{\underline{\chi}}
	 
	\DeclareMathOperator{\Lie}{Lie}

	\newcommand{\lra}{\longrightarrow}    
	
	\nc{\surjto}{\twoheadrightarrow}
	\nc{\ts}{\times}
	\nc{\ds}{\displaystyle}
	\nc{\nd}{\noindent}  
	\nc{\ud}{\underline}
	\nc{\ov}{\overline}
	\nc{\maplra}[1]{\buildrel #1 \over \lra}
	\nc{\mapto}[1]{\buildrel #1 \over \to}
	\nc{\setb}[1]{\{  #1\}}
	\nc{\cHom}{\mathcal{H}om}

	\def\a{\alpha}
	\def\b{\beta}

	\def\g{\gamma} \def\G{\Gamma}

	\def\m{\mu}

	 \def\O{\Omega}

	\def\C{\mathbb{C}}
	
	\def\Z{\mathbb{Z}}

	\def\ch{\hbox{\it ch}}

	\newcommand{\downtriangleBase}{%
		\begin{tikzpicture}%
			\draw[line cap=round,-] (0,1.6ex) -- (1.7ex,1.6ex);%
			\draw[line cap=round,-] (0,1.6ex) -- (0.85ex,0);%
			\draw[line cap=round,-] (0.85ex,0) -- (1.7ex,1.6ex);%
		\end{tikzpicture}%
	}
	\DeclareMathOperator{\downtriangle}{\downtriangleBase}

\begin{document}
\title[A perturbative algorithm for flat F-manifolds associated with LG models] {A perturbative algorithm for flat F-manifolds associated with Landau-Ginzburg models}

\author{Jeehoon Park}
\address{Jeehoon Park: QSMS, Seoul National University, 1 Gwanak-ro, Gwanak-gu, Seoul 08826, South Korea}
\email{jpark.math@gmail.com}

\author{Jaewon Yoo}
\address{Department of Mathematics, POSTECH (Pohang University of Science and Technology), San 31, Hyoja-Dong, Nam-Gu, Pohang, Gyeongbuk, 790-784, South Korea. }
\email{yooj1215@postech.ac.kr}


\begin{abstract}

We develop a perturbative algorithm for constructing formal flat $F$-manifold structures on the cohomologies of dGBV (differential Gerstenhaber-Batalin-Vilkovisky) algebras associated with Landau-Ginzburg models. As an application, this approach provides a perturbative construction of formal flat $F$-manifold structures on two important objects: the Jacobian algebra of a homogeneous polynomial with an isolated singularity at the origin, and the primitive cohomology of smooth projective Calabi-Yau complete intersections.

\end{abstract}

\maketitle
\tableofcontents


\section{Introduction}

Let $M$ be a complex manifold with the structure sheaf $\cO_M$. A Frobenius manifold structure on $M$ is a commutative algebra structure on the holomorphic tangent bundle $\mathcal{T}_M$ with a metric satisfying certain compatibility conditions: see Definition \ref{Def_Frob}. Such a structure was first axiomatized by Dubrovin in \cite{Dubrovin} and its first example was given by K. Saito (the universal unfolding of an isolated hypersurface singularity) in \cite{Saito}, which was further studied by M. Saito in \cite{MSaito}. It also plays an important role in formulating the mirror symmetry conjecture: for example, see \cite{BK} and \cite{CS99}. 
Hertling and Manin introduced weaker but still very useful notions: an $F$-manifold (\cite{HM}, \cite[Chapter 2]{Her02}), which is essentially a Frobenius manifold structure without a metric, and a flat $F$-manifold (\cite{Ma05}), which is an $F$-manifold with a flat structure.

One can also think of formal versions of Frobenius manifold structures and flat $F$-manifold structures. In this paper, we will provide a new construction of formal flat $F$-manifold structures on the cohomology of a dGBV (differential Gerstenhaber-Batalin-Vilkovisky) algebra associated to the LG (Landau-Ginzburg) model.
Our construction is explicitly perturbative and algorithmic, so it can be implemented in a computer program using the Gr\"obner basis of the polynomial ring. 

Let us
\begin{equation*}
    A:=\C[\ud x]=\C[x_1, \ldots, x_n]
\end{equation*}
be a polynomial ring and $S=S(\ud x)\in A$ be an arbitrary polynomial.
Then the triple $(\C^n, \O, S)$, where $\O=dx_1 \wedge \cdots \wedge dx_n$ is a holomorphic volume form on the non-compact complex manifold $\C^n$ (non-compact Calabi-Yau manifold), is called the (algebraic) Landau-Ginzburg $B$-model.
There is a natural dGBV (differential Gerstenhaber-Batalin-Vilkovisky) algebra (see Definition \ref{bvd}) associated to $(\C^n, \O, S)$: 
\begin{equation*}
    (\cA, \delta_S, \Delta)
\end{equation*}
where
$\cA:=\mathrm{PV}_{alg}(\C^n)$ is the space of algebraic polyvector fields on $\C^n$, $\delta_S$ is the twist by the holomorphic 1-form $dS$, and $\Delta$ is the divergence operator with respect to $\O$:
\begin{equation}\label{dgbv}
	\begin{aligned}
	\mathcal{A}&= {\mathbb{C}}[\underline{x}][\underline{\eta}]={\mathbb{C}}[x_1,\ldots,x_n][\eta_{1},\ldots,\eta_n], \quad \eta_i = \frac{\partial}{\partial x_i}\\
	\delta_S&={\{S, -\}}=\sum_{i=1}^n \prt{S(\underline{x})}{x_i} \pa{\eta_i}: \mathcal{A} \to \mathcal{A},\\
    \Delta&=\sum_{i=1}^n\pa{x_i}\pa{\eta_i}:\mathcal{A}\to \mathcal{A}.
	\end{aligned}
\end{equation}

The additive cohomological $\bZ$-grading of $\mathcal{A}$ is given
by the rules ($|f|=m$ means $f \in \mathcal{A}^m$)
\[
	|x_i|=0, |\eta_i|=-1, \quad i=1, \dots, n;
\]
and we have the following cochain complex\footnote{Note that the super-commutativity means that $a\cdot b=(-1)^{|a| |b|} b \cdot a$ for homogeneous elements $a, b$. Hence, we see that $\eta_i \cdot \eta_j = - \eta_j \cdot \eta_i$, which implies that $\eta_i^2 =0$ and $\mathcal{A}^{-n-1}=\mathcal{A}^{-n-2}=\cdots=0$. }
\[
	\cA=\bigoplus_{k=-n}^0 \cA^k, \quad 0 \to \mathcal{A}^{-n} \maplra{\delta_S} \mathcal{A}^{-n+1} \maplra{\delta_S} \cdots \maplra{\delta_S} \mathcal{A}^0=A \to 0.
\]

Note that the cohomology $H(\cA, \delta_S)$ of $(\cA,\delta_S)$ has an induced super-commutative algebra structure from $\cA$, since $\delta_S$ is a derivation.
Let $J_S$ be any finite-dimensional (as $\C$-vector spaces) subalgebra of $H^0(\cA, \delta_S)$ and we assume that
\begin{equation}\label{rassume}
    J_S = \cB^0/\delta_S(\cB^{-1}), \quad \cB= \bigoplus_i \cB^i
\end{equation}
where $\cB$ is a super-commutative subalgebra of $\cA$ and $(\cB, \delta_S, \Delta)$ is again a dGBV algebra.
We would like to provide a perturbative algorithm for formal flat $F$-manifold structures on $J_S$.

\begin{example}
    We give two key examples:
    \begin{enumerate}
        \item Assume that $S(\ud x)$ has only an isolated singularity at zero (this is called the isolated hypersurface singularity). Then we let
        \begin{equation}
            J_S:= H^0(\cA, \delta_S),
        \end{equation}
        which is finite dimensional by the property that the singularities are isolated. Moreover, it is well-known (for example, see \cite{LLS}) that $H^0(\cA, \delta_S)= H(\cA, \delta_S)$. In other words, we have
        \begin{equation*}
            J_S = A/ \mathrm{Jac}(S) = H(\cA, \delta_S)
        \end{equation*}
        where $\mathrm{Jac}(S)$ is the Jacobian ideal of $S$.
        \item Let $\mathbf{P}^N$ be a $N$-dimensional projective space over $\mathbb{C}$ for $N\geq 1$. Denote by $\mathbb{C}[\underline{z}]$ the usual homogeneous coordinate ring of $\mathbf{P}^N$ with $\underline{z} = (z_0, z_1, \dots, z_n)$. For $N\geq k \geq 1$, let $G_1(\underline{z}), \dots, G_k(\underline{z})$ be homogeneous polynomials of degree $d_1, \dots, d_k$ respectively. We consider a smooth projective variety $X_{\underline{G}}$ embedded in $\mathbf{P}^N$ defined by $G_1(\underline{z}), \dots, G_k(\underline{z})$, which satisfies the complete intersection property, i.e. $\dim X_{\underline{G}} = N-k$. We further assume that $X$ is Calabi-Yau:
        \begin{equation*}
            N+1 = \sum_{i=1}^k d_k.
        \end{equation*}
Let $X=X_{\underline{G}}(\mathbb{C})$ be the complex analytic manifold associated to $X_{\underline{G}}$ and consider the Dwork potential
\begin{equation}\label{DP}
	S(\underline{y},\underline{z}) := \sum_{\ell=1}^k y_\ell \cdot G_\ell(\underline{z}),
\end{equation}
where we introduce the formal variables $y_1,\dots,y_k$ corresponding to $G_1,\dots,G_k$. Let $n=N+k+1$, $x_1=y_1,\dots,x_k=y_k,x_{k+1}=z_0,\dots,x_n=z_N$. 
For each  $x_\m$, assign a non-zero integer $\ch(x_\m)$ called
charge of $x_\m$ as follows:
\[
	\ch(x_i)=\ch(y_i)=-d_i, \quad \text{for } i=1, \dots, k, \quad \ch(x_i)=\ch(z_{i-k-1}) = 1, \quad \text{for }  i=k+1, \dots, n. 
\]
Also  assign $ \ch(\eta_\m):=-\ch(x_\m)$.
Then $\cA_0$, the charge zero component of $\cA$, is a super-commutative subalgebra of $\cA$ and $(\cA_0, \delta_S, \Delta)$ is a dGBV algebra.
Let
\begin{equation}
    J_S:=H^0(\cA_0, \delta_S) \subset H^0(\cA, \delta_S).
\end{equation}
Then it is well-known (\cite{Gr69}, \cite{Ko91}, and \cite{Dim95}) that $J_S$ is isomorphic to the primitive (middle-dimensional) cohomology of $X$. In particular, $J_S$ is finite dimensional over $\C$.
    \end{enumerate}
\end{example}

\textbf{Acknowledgement:}
The work of Jeehoon Park was supported by the National Research Foundation of Korea (NRF-2021R1A2C1006696) and the National Research Foundation of Korea (NRF) grant funded by the Korea government (MSIT) (No.2020R1A5A1016126).

The first author would like to thank Jae-Suk Park for helping him initiating this project and providing useful comments and warm support. He is also grateful to Younggi Lee and Jaehyun Yim for helpful discussions on the subject. 




\section{DGBV algebras and flat $F$-manifolds}

 We briefly review the definitions of dGBV-algebras.

\begin{definition}\label{bvd}
	Let $k$ be a field. Let $(\cC,\cdot)$ be a unital $\Z$-graded super-commutative and associative $k$-algebra.
	Let $[\cdot, \cdot]: \cC \otimes \cC \to \cC$ be a bilinear map of degree 1.
	\begin{enumerate}[(a)]
		\item $(\cC, \cdot, [\cdot, \cdot])$ is called a G-algebra (Gerstenhaber algebra) over $k$ if
		\begin{align*}
			\quad [a, b] &= (-1)^{|a||b|}[b,a],\\
			[a, [b,c]]&= (-1)^{|a|+1}[[a,b],c]+(-1)^{(|a|+1)(|b|+1)} [b, [a,c]],\\
			\quad [a,b \cdot c]&= [a, b] \cdot c +(-1)^{(|a|+1)\cdot |b|} b \cdot [a,c],
		\end{align*}
		for any homogeneous elements $a, b, c \in \cC.$
		\item  $(\cC,\cdot, \Delta)$ is called a BV algebra, 
		if $(\cC, \Delta, \ell_2^{\Delta})$ is a shifted DGLA(differential graded Lie algebra), where
		\begin{align}\label{elltwo}
			\ell_2^{\Delta}(a,b):= \Delta(a \cdot b)-\Delta(a)\cdot b -(-1)^{|a|} a\cdot \Delta(b), \quad a,b \in \cC,
		\end{align} 
        and $(\cC, \cdot, \ell_2^{\Delta})$ is a G-algebra, 
		\item $(\cC, \cdot, \Delta, \delta)$, where $\delta:\cC \to \cC$ is a linear map of degree 1, is called a dGBV(differential Gerstenhaber-Batalin-Vilkovisky) algebra if $(\cC, \cdot, \Delta,  \ell_2^{\Delta})$ is a GBV algebra
		and $(\cC,\cdot,\delta)$ is a cdga(commutative differential graded algebra), i.e.
		\[
			\delta^2=0, \quad \delta(a \cdot b) = \delta(a) \cdot b + (-1)^{|a|} a \cdot \delta(b), \quad a,b \in \cC,
		\]
		 and $(\delta+\Delta)^2=0$.
	 \end{enumerate}
\end{definition}

We now review the notions of flat $F$-manifolds and Frobenius manifolds.
Let $M$ be a connected complex manifold with the holomorphic structure sheaf $\cO_M$. Let $\cT_M$ be the holomorphic tangent sheaf.
A $(k,l)$-tensor means an $\cO_M$-linear map $T: \cT_M^{\otimes k} \to \cT_M^{\otimes l}$.
The Lie derivative $\Lie_X$ along a vector field $X$ is a derivation on the sheaf of $(k,l)$-tensors, as well as the covariant derivative $\nabla_X T$ with respect to a connection $\nabla$ on $M$:
\be
&&\Lie_X (Y_1 \otimes \ldots \otimes Y_l) = \sum_i Y_1 \otimes \ldots \Lie_X(Y_i) \ldots \otimes Y_l, \\
&&
\nabla_X (Y_1 \otimes \ldots \otimes Y_l) = \sum_i Y_1 \otimes \ldots \nabla_X(Y_i) \ldots \otimes Y_l 
\ee
for local vector fields $Y_1, \ldots, Y_l$, and
\be
(\Lie_X T)(Y) = \Lie_X (T(Y)) - T(\Lie_X(Y)), \quad
(\nabla_X T)(Y) = \nabla_X (T(Y)) - T(\nabla_X(Y))
\ee
for $(k,l)$-tensors $T$. Note that $\Lie_X(f)=X(f)$ for functions $f$ and $\Lie_X Y= [X,Y]$ for vector fields $X,Y$.
Thus the Lie derivative $\Lie_X T$ of a $(k,l)$-tensor $T$ along a vector field $X$ is again a $(k,l)$-tensor, as well as the covariant derivative $\nabla_X T$ with respect to a connection $\nabla$ on $M$.
Then $\nabla T$ can be viewed as a $(k+1,l)$-tensor.

\begin{definition} [$F$-manifolds]
An $F$-manifold is a triple $(M, \circ, e)$ where $\circ$ is a commutative and associative $\cO_M$-bilinear multiplication $\cT_M \times \cT_M \to \cT_M$, $e$ is a global unit vector field with respect to $\circ$, and the multiplication satisfies
\bea \label{F-pot}
\Lie_{X\circ Y}(\circ) = X \circ \Lie_Y(\circ) + Y\circ \Lie_X(\circ)
\eea
for any local vector fields $X,Y$. 
\end{definition}
The condition \eqref{F-pot} is equivalent to
		\begin{multline}\label{fmfd}
			[X \circ Y, Z \circ W] - [X \circ Y,Z] \circ W - Z \circ [X \circ Y,W] - X \circ [Y, Z \circ W] + X \circ [Y,Z] \circ W\\
			+ X \circ Z \circ [Y,W] - Y \circ [X, Z \circ W] + Y \circ [X,Z] \circ W + Y \circ Z \circ [X,W] = 0
		\end{multline}
for any local vector fields $X,Y,Z,W$.

\begin{definition}[flat $F$-manifolds] \label{ffm}
A flat $F$-manifold is a quadruple $(M,\circ, e, \nabla)$ where $\circ$ is a commutative and associative $\cO_M$-bilinear multiplication $\cT_M \times \cT_M \to \cT_M$, $e$ is a global unit vector field with respect to $\circ$, $\nabla$ is a flat connection on $\cT_M$ which satisfies $\nabla e=0$, and $\circ$ is compatible with $\nabla$, i.e. each element of the pencil $(\nabla^z)_{z\in \C}$, defined by $\nabla_X^z(Y) = \nabla_X Y + z X\circ Y$ is torsion-free and flat:
\be
&& \nabla^z_{X}Y -\nabla^z_{Y}X=[X, Y] \quad (\text{torsion-free}),\\
&&\nabla^z_{X} \nabla^z_{Y} -\nabla^z_{Y} \nabla^z_{X} =\nabla^z_{[X, Y]} \quad (\text{flat}).
\ee
\end{definition}
These axioms imply an existence of a vector potential in $\nabla$-flat coordinates: in $\nabla$-flat local coordinates $t_M=\{t^0, \ldots, t^{\mu-1}\}$ on $M$ with $\mu=\dim M$, there exists a vector potential $(F^0, \ldots, F^{\mu-1})$ such that
\be
\frac{\partial}{\partial t^\a} \circ \frac{\partial}{\partial t^\b} = \sum_{\g} \frac{\partial^2 F^\g}{\partial t^\a \partial t^\b} \frac{\partial}{\partial t^\g}, \quad \a, \b = 0, \ldots, \mu-1.
\ee

\begin{definition}[Frobenius manifolds]\label{Def_Frob}
A Frobenius manifold is a tuple $(M, \circ, e, g)$ where $g$ is a metric on $M$, $\circ$ is a commutative and associative $\cO_M$-bilinear multiplication $\cT_M \times \cT_M \to \cT_M$, and $e$ is a global unit vector field with respect to $\circ$ subject to the following conditions:
\begin{enumerate}[(1)]
\item (invariance) $g(X\circ Y, Z) = g(X, Y\circ Z)$
\item (potentiality) the (3,1)-tensor $\nabla^g \circ$ is symmetric where $\nabla^g$ is the Levi-Civita connection of $g$,
\item the metric $g$ is flat, i.e. $\nabla^g$ is a flat connection, $[\nabla^g_X,\nabla^g_Y]=\nabla^g_{[X,Y]}$,
\item $\nabla^g e=0$.
\end{enumerate}
\end{definition}

The potentiality condition (2) written out for arbitrary local fields $X,Y,Z$ is
\be
\nabla^g_X(Y\circ Z) - Y \circ \nabla_X^g(Z) - \nabla^g_Y (X \circ Z) + X \circ \nabla^g_Y(Z) - [X,Y]\circ Z=0.
\ee

If $\ud t=\{t^\a\}$ is a $\nabla^g$-flat coordinate on $M$, then there is a potential $F(\ud t)$ such that
\be
\partial_\a \partial_\b \partial_\g F(\ud t) = g(\partial_\a \circ \partial_\b, \partial_\g), \quad
\text{ where } \partial_\a =\frac{\partial}{\partial t^\a}.
\ee
The following proposition is well-known (\cite{Her02}).
\begin{proposition}
If $(M, \circ, e,g)$ is a Frobenius manifold, then $(M, \circ, e, \nabla^g)$ is a flat $F$-manifold.
If $(M,\circ, e, \nabla)$ is a flat $F$-manifold, then $(M, \circ, e)$ is an $F$-manifold.
\end{proposition}
One can similarly define formal versions of $F$-manifolds, flat $F$-manifolds, and Frobenius manifolds: we can consider the formal structure sheaf and the formal tangent bundle instead of the holomorphic structure sheaf and the holomorphic tangent bundle.

\section{Flat $F$-manifolds and differential equations}

In \cite{ST}, Saito-Takahashi explained how to deduce a Frobenius manifold structure from a primitive form.
Recall \eqref{rassume} and let $\ud t=\{t^\alpha:\alpha\in I\}$ be a coordinate system of the affine manifold $J_S=\cB^0/\delta_S(\cB^{-1}) \subset H^0(\cA, \delta_S)$. Here $I$ is a finite index set.
We choose a $\C$-basis $\{[u_\alpha]=u_\a + \delta_S(\cB^{-1}):u_\a \in \cB^0, \alpha\in I\}$ of $J_S$ and assign an additive cohomological grading 
\begin{equation}
    |t^\a|=0, \quad \a \in I; \quad \cB[[\ud t]]=\bigoplus_{i} (\cB[[\ud t]])^i=\bigoplus_i \cB^i[[\ud t]].
\end{equation}

Then, for any $\Gamma\in \cB^0[\![\underline{t}]\!]$, the triple $(\cB[[\ud t]], \delta_{S+\G}, \Delta)$ is again a dGBV algebra over $\C[[\ud t]]$, where
\[
\delta_{S+\G}={\{S+\G, -\}}=\sum_{i=1}^n \prt{(S+\G)}{x_i} \pa{\eta_i}: \mathcal{B}[[\ud t]] \to \mathcal{B}[[\ud t]]
\]
From now on we assume that $\Gamma\in \cB^0[\![\underline{t}]\!]$ satisfies $\partial_\alpha\Gamma|_{\underline{t}=0}=u_\alpha$ for all $\alpha\in I$.

\begin{definition}\label{H-module}
	For $\Gamma\in \cB^0[\![\ud t]\!]$ and a formal variable $\hbar$ with $|\hbar|=0$, we define 
	\[
		\mathcal{H}_{S+\Gamma} :=\frac{\cB^0[\![\underline{t}]\!](\!(\hbar)\!)}{(\delta_{S+\Gamma}+\hbar\Delta)(\cB^{-1}[\![\underline{t}]\!](\!(\hbar)\!))},\quad\mathcal{H}_{S+\Gamma}^{(m)}:=\frac{\cB^0[\![\underline{t}]\!][\![\hbar]\!]\hbar^{-m}}{(\delta_{S+\Gamma}+\hbar\Delta)(\cB^{-1}[\![\underline{t}]\!][\![\hbar]\!]\hbar^{-m})}
\]
	for $m\in\mathbb{Z}$. 
\end{definition}

We define the Gauss-Manin connection on $\mathcal{H}_{S+\Gamma}$ following \cite[Definition 4.4]{ST}. We use the notation $[\cdot]$ to denote the cohomology class.
		 For each $\alpha\in I$, we define a connection $\downtriangle_\alpha^{\frac{S+\Gamma}{\hbar}}:\mathcal{H}_{S+\Gamma}\to \mathcal{H}_{S+\Gamma}$ by
		\begin{align*}
			\downtriangle_\alpha^{\frac{S+\Gamma}{\hbar}}([w]):=&\left[e^{\tfrac{-(S+\Gamma)}{\hbar}}\frac{\partial}{\partial t^\alpha}\left(e^{\tfrac{S+\Gamma}{\hbar}}w\right)\right]\\
			=&\left[\frac{\partial}{\partial t^\alpha}w+\frac{1}{\hbar}\frac{\partial(S+\Gamma)}{\partial t^\alpha}w\right]
		\end{align*}
		for $w\in\cB^0 [\![\underline{t}]\!](\!(\hbar)\!)$.
In the theory of primitive forms, the following differential equation (cf. \cite[equation (S2)]{ST}) plays an important role:
    \begin{equation}\label{WPF}
			\hbar\nabla_\a^{\frac{S+\G}{\hbar}}\hbar\nabla_\b^{\frac{S+\G}{\hbar}}\boldsymbol\zeta=\sum_{\rho}\mathbf{A}_{\alpha\beta}{}^\rho\cdot\hbar\nabla_\rho^{\frac{S+\G}{\hbar}}\boldsymbol\zeta+\big(\delta_{S+\Gamma}+\hbar\Delta\big)(\mathbf{\Lambda}_{\alpha\beta})
    \end{equation}
        where $\boldsymbol\zeta\in\cB^0[\![\underline{t}]\!][\![\hbar]\!]$, $\mathbf{A}_{\alpha\beta}{}^\rho={}^0A_{\alpha\beta}{}^\rho+{}^1A_{\alpha\beta}{}^\rho\hbar\in\mathbb{C}[\![\underline{t}]\!][\![\hbar]\!]$, and
$\mathbf{\Lambda}_{\alpha\beta}\in\cB^{-1}[\![\underline{t}]\!][\![\hbar]\!]$.
Note that $\boldsymbol\zeta$ serves as a primitive form. We sometimes abbreviate the notation $\partial_\a =\frac{\partial}{\partial t_\a}, \a\in I$.
		
\begin{proposition} \label{wpfflat}
Solutions $\boldsymbol\zeta\in\cB^0[\![\underline{t}]\!][\![\hbar]\!]$, $\Gamma\in \cB[\![\ud t]\!]$, and $\mathbf{A}_{\alpha\beta}{}^\rho={}^0A_{\alpha\beta}{}^\rho+{}^1A_{\alpha\beta}{}^\rho\hbar\in\mathbb{C}[\![\underline{t}]\!][\![\hbar]\!]$ to the differential system \eqref{WPF} give rise to the structure of a formal flat $F$-manifold $(M,\circ, e, \nabla)$
where $M=H(\cB, \delta_S)$, $\circ$ is the multiplication induced from $\cB$ (namely, $\partial_\alpha\circ \partial_\beta = \sum_\rho {}^0A_{\alpha\beta}{}^\rho\cdot\partial_\rho$), $e=[1]$, and the connection $\nabla$ is defined by $\nabla_{\partial_\a}\partial_\b = \sum_{\rho} {}^1A_{\alpha\beta}{}^\rho \partial_\rho$.
\end{proposition}

\begin{proof}
Since $\hbar\nabla_\a^{\frac{S+\G}{\hbar}}\hbar\nabla_\b^{\frac{S+\G}{\hbar}} = \hbar\nabla_\b^{\frac{S+\G}{\hbar}}\hbar\nabla_\a^{\frac{S+\G}{\hbar}}$, the multiplication $\circ$ is commutative and $\nabla$ is torsion-free. 
Consider the following equation (with a simplified notation $\downtriangle_\a=\downtriangle_\a^{\frac{S+\G}{\hbar}}$):
\begin{align*}
	\hbar\downtriangle_\gamma\hbar\downtriangle_\beta\hbar\downtriangle_\alpha\boldsymbol\zeta=&\hbar\downtriangle_\gamma \Big[\sum_{\rho\in I}\left(\mathbf{A}_{\alpha\beta}{}^\rho\right)\hbar\downtriangle_\rho\boldsymbol\zeta\Big]+(\delta_{S+\Gamma}+\hbar\Delta)(\hbar\downtriangle_\gamma\mathbf{\Lambda}_{\alpha\beta})\\
	=&\sum_{\rho\in I}\Big[\hbar\partial_\gamma(\mathbf{A}_{\alpha\beta}{}^\rho)\hbar\downtriangle_\rho\boldsymbol\zeta+(\mathbf{A}_{\alpha\beta}{}^\rho)\hbar\downtriangle_\gamma\hbar\downtriangle_\rho\boldsymbol\zeta\Big]+(\delta_{S+\Gamma}+\hbar\Delta)(\hbar\downtriangle_\gamma\mathbf{\Lambda}_{\alpha\beta})\\
	=&\sum_{\delta\in I}\bigg[\hbar\mathbf{A}_{\alpha\beta,\gamma}{}^\delta+\sum_{\rho\in I}\big(\mathbf{A}_{\alpha\beta}{}^\rho\mathbf{A}_{\gamma\rho}{}^\delta\big)\bigg]\hbar\downtriangle_\delta\boldsymbol\zeta\\
	&+(\delta_{S+\Gamma}+\hbar\Delta)\Big(\hbar\downtriangle_\gamma\mathbf{\Lambda}+\sum_{\rho\in I}\mathbf{A}_{\alpha\beta}{}^\rho\mathbf{\Lambda}_{\gamma\rho}\Big).
\end{align*}
Therefore, the following quantities are invariant under the permutation of $\alpha,\beta,\gamma$ for all $\delta$:
\begin{equation}\label{weak_F}
	\begin{aligned}
		\hbar\mathbf{A}_{\alpha\beta,\gamma}{}^\delta+&\sum_{\rho\in I}\big(\mathbf{A}_{\alpha\beta}{}^\rho\mathbf{A}_{\gamma\rho}{}^\delta\big)\\
		=&\sum_{\rho\in I}{}^0A_{\alpha\beta}{}^\rho\cdot{}^0A_{\gamma\rho}{}^\delta+\hbar\Big[{}^0A_{\alpha\beta,\gamma}{}^\delta+\sum_{\rho\in I}{}^0A_{\alpha\beta}{}^\rho\cdot{}^1A_{\gamma\rho}{}^\delta+{}^1A_{\alpha\beta}{}^\rho\cdot{}^0A_{\gamma\rho}{}^\delta\Big]\\
		&+\hbar^2\Big[{}^1A_{\alpha\beta,\gamma}{}^\delta+\sum_{\rho\in I}{}^1A_{\alpha\beta}{}^\rho\cdot{}^1A_{\gamma\rho}{}^\delta\Big].
	\end{aligned}
\end{equation}

The invariance of $\sum_{\rho\in I}{}^0A_{\alpha\beta}{}^\rho\cdot{}^0A_{\gamma\rho}{}^\delta$ under the permutation of indices indicates that $\circ$ is associative.
The invariance of ${}^0A_{\alpha\beta,\gamma}{}^\delta+\sum_{\rho\in I}{}^0A_{\alpha\beta}{}^\rho\cdot{}^1A_{\gamma\rho}{}^\delta+{}^1A_{\alpha\beta}{}^\rho\cdot{}^0A_{\gamma\rho}{}^\delta$ implies that the $(3,1)$-tensor $\nabla \circ$ is symmetric in all three arguments. Thus \eqref{F-pot} holds by using \cite[Theorem 2.14]{Her02}.
Moreover, this invariance of indices also says that 
each element of the pencil $(\nabla^z)_{z\in \C}$, defined by $\nabla_X^z(Y) = \nabla_X Y + z X\circ Y$ is torsion-free and flat:
\be
&& \nabla^z_{X}Y -\nabla^z_{Y}X=[X, Y] \quad (\text{torsion-free}),\\
&&\nabla^z_{X} \nabla^z_{Y} -\nabla^z_{Y} \nabla^z_{X} =\nabla^z_{[X, Y]} \quad (\text{flat}).
\ee
\end{proof}

We analyze \eqref{WPF} when $\boldsymbol\zeta=1 \in \cB$.
Since $\partial_\alpha\Gamma=\hbar\downtriangle_\alpha^{\frac{S+\Gamma}{\hbar}}1$ is a $\mathbb{C}[\![\underline{t}]\!][\![\hbar]\!]$-basis of $\mathcal{H}_{S+\Gamma}^{(0)}$, there is a connection matrix $\mathbf{A}_{\alpha\beta}{}^\rho\in\mathbb{C}[\![\underline{t}]\!][\![\hbar]\!]$ with respect to the basis $\{\hbar\downtriangle_\rho^{\frac{S+\Gamma}{\hbar}}1:\rho\in I\}$ such that
\[
	\hbar\downtriangle_\beta^{\frac{S+\Gamma}{\hbar}}(\hbar\downtriangle_\alpha^{\frac{S+\Gamma}{\hbar}}1)=\sum_{\rho\in I}\mathbf{A}_{\alpha\beta}{}^\rho\cdot(\hbar\downtriangle_\rho^{\frac{S+\Gamma}{\hbar}}1)+(\delta_{S+\Gamma}+\hbar\Delta)(\mathbf{\Lambda}_{\alpha\beta})
\]
for some $\mathbf{\Lambda}_{\alpha\beta}\in\cB^{-1}[\![\underline{t}]\!][\![\hbar]\!]$ and all $\alpha,\beta\in I$. 
Suppose that there is $\Gamma \in \cB^0[[\ud t]]$ that makes $\mathbf{A}_{\alpha\beta}{}^\rho$ and $\mathbf{\Lambda}_{\alpha\beta}$ have no $\hbar$-power terms, i.e., $\ud t= \{t^\alpha\}$ is a $\nabla$-flat coordinate (${}^1A_{\alpha\beta}{}^\rho=0$) and there is a connection matrix $A_{\alpha\beta}{}^\rho={}^0A_{\alpha\beta}{}^\rho \in\mathbb{C}[\![\underline{t}]\!]$ such that
\begin{equation}\label{F-QM}
	\hbar\downtriangle_\beta^{\frac{S+\Gamma}{\hbar}}(\hbar\downtriangle_\alpha^{\frac{S+\Gamma}{\hbar}}1)=\sum_{\rho\in I}A_{\alpha\beta}{}^\rho\cdot(\hbar\downtriangle_\rho^{\frac{S+\Gamma}{\hbar}}1)+(\delta_{S+\Gamma}+\hbar\Delta)(\Lambda_{\alpha\beta}),
\end{equation}
for some $\Lambda_{\alpha\beta}\in\cB^{-1}[\![\underline{t}]\!]$ and all $\alpha,\beta\in I$. Then, by comparing $\hbar$-power terms, the equation \eqref{F-QM} reduces further to the following:
\begin{equation}\label{F-QM11}
	\begin{aligned}
		\partial_\alpha\Gamma\cdot\partial_\beta\Gamma&=\sum_{\rho\in I}A_{\alpha\beta}{}^\rho\cdot\partial_\rho\Gamma+\delta_{S+\Gamma}(\Lambda_{\alpha\beta}),\\
		\partial_\beta\partial_\alpha\Gamma&=\Delta(\Lambda_{\alpha\beta}),
	\end{aligned}
\end{equation}
for all $\alpha,\beta\in I$. 
By Proposition \ref{wpfflat}, solutions of \eqref{F-QM11} give a formal flat $F$-manifold structure on $J_S$. For simplicity, we use the notation $\Gamma_{\alpha_1\alpha_2\cdots\alpha_\ell}:=\partial_{\alpha_\ell}\cdots\partial_{\alpha_2}\partial_{\alpha_1}\Gamma$. Note that equation \eqref{F-QM11} is similar to the equation appeared in the final part of the proof of \cite[Lemma 7.1]{BK}.


\section{An explicit algorithm with flat coordinates}\label{sec3.4}
This section is the key point of the current article. Now we provide an algorithm for a solution to the differential equation \eqref{F-QM11}.
Recall that $\partial_\alpha\Gamma|_{\underline{t}=0}=u_\alpha$ for all $\alpha\in I$, where $\{[u_\alpha]:\alpha\in I\}$ is a $\mathbb{C}$-basis of $J_S$.
We use the following notation:
\begin{equation}\label{Exp_notation}
	\begin{aligned}
		\Gamma&=\sum_{\alpha\in I}u_\alpha\cdot t^\alpha+\sum_{m\geq2}\sum_{\underline{\alpha}\in I^m}\frac{1}{m!}u_{\underline{\alpha}}t^{\underline{\alpha}}\in \cB^0 [\![\underline{t}]\!],\\
		A_{\alpha\beta}{}^\rho&=a_{\alpha\beta}{}^\rho+\sum_{m\geq1}\sum_{\underline{\alpha}\in I^m}\frac{1}{m!}a_{\alpha\beta\underline{\alpha}}{}^\rho t^{\underline{\alpha}}\in\mathbb{C}[\![\underline{t}]\!],\\
		\Lambda_{\alpha\beta}&=\lambda_{\alpha\beta}+\sum_{m\geq1}\sum_{\underline{\alpha}\in I^m}\frac{1}{m!}\lambda_{\alpha\beta\underline{\alpha}}t^{\underline{\alpha}}\in\cB^{-1} [\![\underline{t}]\!],
	\end{aligned}
\end{equation}
where $u_{\underline{\alpha}}\in\cB^0$, $a_{\alpha\beta}{}^\rho,a_{\alpha\beta\underline{\alpha}}{}^\rho\in\mathbb{C}$, $\lambda_{\alpha\beta},\lambda_{\alpha\beta\underline{\alpha}}\in\cB^{-1}$, and $t^{\alpha_1\cdots\alpha_m}:=t^{\alpha_1}\cdots t^{\alpha_m}$.

We assume that $u_{\underline{\alpha}},a_{\alpha\beta\underline{\alpha}}{}^\rho,\lambda_{\alpha\beta\underline{\alpha}}$ are invariant under the permutation of indices of $\underline{\alpha}$ under our notation. This assumption is important for our algorithm (see Remark \ref{nontrivialE}). 
We will provide a perturbative explicit algorithm which calculates $\Gamma, A_{\alpha\beta}{}^\rho$ satisfying equation \eqref{F-QM11}, i.e., we have to determine $u_{\underline{\alpha}},a_{\alpha\beta}{}^\rho, $and $ a_{\alpha\beta\underline{\alpha}}{}^\rho$ which satisfy \eqref{F-QM11}.\par 

For $[u] \in J_S$ with $u \in \cB$,
    \begin{equation}\label{ddbar}
        \text{ we assume that $\delta_S(u)=0$ implies that $\Delta(u)$ belongs to the image of $\Delta \circ \delta_S$. }
    \end{equation}
This assumption is a weak analogue of the $\partial \overline{\partial}$-lemma for compact K\"{a}hler manifolds. Our algorithm works without this assumption, but it guarantees the uniqueness of a solution (see Corollary \ref{Cor_Alg}).
%

\begin{definition}\label{u_alpha}
	Let $|\underline{\alpha}|=m$ which means that $\underline{\alpha}\in I^m$. We define the notation $u_{\underline{\alpha}}^{(i)}$ as follows:
	\[
		u^{(i)}_{\underline{\alpha}}=\sum_{\substack{\underline{\alpha}_1\sqcup\cdots\sqcup \underline{\alpha}_{m-i}=\underline{\alpha}\\\underline{\alpha}_j\neq\emptyset}}\frac{1}{(m-i)!}u_{\underline{\alpha}_1}\cdots u_{\underline{\alpha}_{m-i}},\quad(0\leq i\leq m-1),	
	\]
	where the notation $\underline{\alpha}_1\sqcup\cdots\sqcup \underline{\alpha}_{m-i}=\underline{\alpha}$ means $\underline{\alpha}_1\cup\cdots\cup \underline{\alpha}_{m-i}=\underline{\alpha}$ and $\underline{\alpha}_k\cap \underline{\alpha}_\ell=\emptyset$ for $\ell\neq k$.
\end{definition}
Note that $u_{\underline{\alpha}}^{(i)}$ is invariant under the permutation of indices of $\underline{\alpha}$. For example,
\[
	\begin{aligned}
		u_{\alpha\beta\gamma}^{(0)}&=u_\alpha u_\beta u_\gamma,\\
		u_{\alpha\beta\gamma}^{(1)}&=u_\alpha u_{\beta\gamma}+ u_{\beta} u_{\alpha\gamma}+ u_\gamma u_{\alpha\beta},\\
		u_{\alpha\beta\gamma}^{(2)}&=u_{\alpha\beta\gamma},
	\end{aligned}
	\quad\textrm{and}\quad
	\begin{aligned}
		u_{\alpha\beta\gamma\delta}^{(0)}=&u_\alpha u_\beta u_\gamma u_\delta,\\
		u_{\alpha\beta\gamma\delta}^{(1)}=&u_{\alpha\beta}u_{\gamma}u_{\delta}+u_{\alpha\gamma}u_{\beta}u_{\delta}+u_{\alpha\delta}u_{\beta}u_{\gamma}\\
		&+u_{\beta\gamma}u_{\alpha}u_{\delta}+u_{\beta\delta}u_{\alpha}u_{\gamma}+u_{\gamma\delta}u_{\alpha}u_{\beta},\\
		u_{\alpha\beta\gamma\delta}^{(2)}=&u_{\alpha\beta} u_{\gamma\delta}+ u_{\alpha\gamma} u_{\beta\delta}+ u_{\alpha\delta} u_{\beta\gamma}\\
		&+u_\alpha u_{\beta\gamma\delta}+ u_{\beta} u_{\alpha\gamma\delta}+ u_\gamma u_{\alpha\beta\delta}+ u_{\delta} u_{\alpha\beta\gamma},\\
		u_{\alpha\beta\gamma\delta}^{(3)}=&u_{\alpha\beta\gamma\delta}.
	\end{aligned}
\]\par

The structure of the algorithm is as follows.
\begin{itemize}
\item \textbf{Input of the algorithm}:
Choose a $\mathbb{C}$-basis $\{[u_\alpha]=u_\a + \delta_S(\cB^{-1}):\alpha\in I\}$ of $J_S =H^0(\cB, \delta_S)$. 
\item \textbf{Output of the algorithm}:
We obtain an algebraic perturbative expansion of $\G$ and $A_{\alpha\beta}{}^\rho$ which provides the multiplication $\circ$ on $J_S$, i.e. the formal flat $F$-manifold $(J_S, \circ, [1], \nabla)$ where we use a $\nabla$-flat coordinate $\ud t$ so that $\nabla_{\partial_\a} \partial_\b =0$.
\end{itemize}

Now we give the algorithm for $\Gamma$ and $A_{\alpha\beta}{}^\rho$ as follows:
\begin{description}
	\item[Step-1] Determine $a_{\alpha\beta}^{(0)}$ and $\lambda_{\alpha\beta}^{(0)}$ using a basis $\{[u_\alpha]:\alpha\in I\}$ of $J_S$ as follows:
	\[
		u_\alpha u_\beta=\sum_{\rho\in I}a_{\alpha\beta}^{(0)}{}^\rho u_\rho+\delta_S({\lambda_{\alpha\beta}^{(0)}}).
	\]
	Note that $a_{\alpha\beta}^{(0)}{}^\rho\in\mathbb{C}$ is unique and $\lambda_{\alpha\beta}^{(0)}$ is unique up to $\ker \delta_S$. 
    Under the assumption \eqref{ddbar}, $\Delta(\lambda_{\alpha\beta}^{(0)})$ is unique up to $\mathrm{Im}(\delta_S)$.

    Then, define $u_{\alpha\beta}:=\Delta(\lambda_{\alpha\beta}^{(0)})$ and $a_{\alpha\beta}{}^\rho:=a_{\alpha\beta}^{(0)}{}^\rho$. 
	\item[Step-$\boldsymbol\ell$ $(\ell\geq2)$] Suppose that $|\underline{\alpha}|=\ell+1$. Determine $a_{\underline{\alpha}}^{(i)}{}^\rho$ and $\lambda_{\underline{\alpha}}^{(i)}$ ($0\leq i\leq \ell-1$) in sequence as follows:
	\begin{equation}\label{ind_u}
			\begin{aligned}
				u_{\underline{\alpha}}^{(0)}&=\sum_{\rho\in I}a_{\underline{\alpha}}^{(0)}{}^\rho u_\rho+\delta_S(\lambda_{\underline{\alpha}}^{(0)}),\\
				u_{\underline{\alpha}}^{(1)}-\Delta(\lambda_{\underline{\alpha}}^{(0)})&=\sum_{\rho\in I}a_{\underline{\alpha}}^{(1)}{}^\rho u_\rho+\delta_S(\lambda_{\underline{\alpha}}^{(1)}),\\
				&\quad\vdots\\
				u_{\underline{\alpha}}^{(\ell-1)}-\Delta(\lambda_{\underline{\alpha}}^{(\ell-2)})&=\sum_{\rho\in I}a_{\underline{\alpha}}^{(\ell-1)}{}^\rho u_\rho+\delta_S(\lambda_{\underline{\alpha}}^{(\ell-1)}).
		\end{aligned}
	\end{equation}
	Since $\lambda_{\underline{\alpha}}^{(i)}$ is unique up to $\ker \delta_S$ for $0\leq i\leq\ell-2$, the assumption \eqref{ddbar} implies that $\Delta(\lambda_{\underline{\alpha}}^{(i)})$ is unique up to $\mathrm{Im}(\delta_S)$. Therefore, $a_{\underline{\alpha}}^{(\ell-1)}{}^\rho$ is independent of the choices of $\lambda_{\underline{\alpha}}^{(i)}$. Then, define $a_{\underline{\alpha}}{}^\rho:=a_{\underline{\alpha}}^{(\ell-1)}{}^\rho$ and $u_{\underline{\alpha}}:=\Delta(\lambda_{\underline{\alpha}}^{(\ell-1)})$.
\end{description}
By this inductive algorithm, we can completely determine $\Gamma, A_{\alpha\beta}{}^\rho$, and $\Lambda_{\a\b}$ which turn out to satisfy the equations \eqref{F-QM11}.
The following corollary is clear from the algorithm.
\begin{corollary}\label{Cor_Alg}
	Under the assumption \eqref{ddbar}, the result $A_{\alpha\beta}{}^\rho(\ud t) \in \C[[\ud t]]$ of the above algorithm depends only on the choice of $\mathbb{C}$-basis $\{[u_\alpha]:\alpha\in I\}$ of $J_S$.
\end{corollary}

\begin{remark}
Since $J_S=H^0(\cB, \delta_S)$ is $\cB^0 \subset \C[\ud x]$ modulo $\delta_S(\cB^{-1})$, the division algorithm based on the Gr\"obner basis for the ideal $(\frac{\partial S}{\partial x_1}, \cdots,\frac{\partial S}{\partial x_n})=\delta_S(\cB^{-1})$ of the polynomial subring $\cB^0$ makes it possible to implement the inductive procedure described in \eqref{ind_u} in a computer program.
\end{remark}

\begin{remark}[Non-triviality of solving the differential equations \eqref{F-QM11}]\label{nontrivialE}
	Symmetries among the indices of $\underline{\alpha}$ in $u_{\underline{\alpha}},a_{\alpha\beta\underline{\alpha}}{}^\rho,\lambda_{\alpha\beta\underline{\alpha}}$ play an important role. If one tries to solve the equations in a naive way by comparing the coefficients of the $\ud t$-powers, then one would get into trouble because the symmetries among the indices are not guaranteed. But our algorithm guarantees that relevant quantities such as $u_{\underline{\alpha}},a_{\alpha\beta\underline{\alpha}}{}^\rho,\lambda_{\alpha\beta\underline{\alpha}}$ are invariant under the permutation of $\ud \a$ as the equations \eqref{ind_u} indicate. Suppose that there is no condition such as $u_{\alpha\beta}=u_{\beta\alpha}$, that is, we do not put symmetry restrictions on the index notation. Then, the equation
	\[
		\partial_\beta\partial_\alpha\Gamma=\Delta(\Lambda_{\alpha\beta})
	\]
	in the equations \eqref{F-QM11} gives the following equation
	\[
		u_{\alpha\beta}+u_{\beta\alpha}=2\Delta(\lambda_{\alpha\beta}),
	\]
	but then it cannot determine the values of $u_{\alpha\beta},u_{\beta\alpha}$ individually from the data $2\Delta(\lambda_{\alpha\beta})$ of the previous step. This implies that the naive way of solving  \eqref{F-QM11} inductively with non-symmetric notations does not work well, either.
\end{remark}

\section{A proof why the algorithm works}\label{sec3.5}
Here we give a proof of the algorithm in Section \ref{sec3.4}.
\begin{definition}\label{Def_UBC}
	We define quantities $\mathbf{U},\mathbf{B},\mathbf{C}$, which are our key players in the proof of the algorithm, as follows:
	\[
		\mathbf{U}_{\alpha_1\cdots\alpha_m}:=\hbar\downtriangle_{\alpha_m}^{\frac{S+\Gamma}{\hbar}}\cdots\hbar\downtriangle_{\alpha_1}^{\frac{S+\Gamma}{\hbar}}1,
	\]
	\[
		\left\{\begin{aligned}
			{\mathbf{B}_{\alpha_1\alpha_2}}^\rho:=&{A_{\alpha_1\alpha_2}}^\rho,\\
			{\mathbf{B}_{\alpha_1\alpha_2\alpha_3}}^\rho:=&\sum_{\delta\in I}{\mathbf{B}_{\alpha_1\alpha_2}}^\delta{\mathbf{B}_{\delta\alpha_3}}^\rho+\hbar\cdot\partial_{\alpha_3}({\mathbf{B}_{\alpha_1\alpha_2}}^\rho),\\
			&\vdots\\
			{\mathbf{B}_{\alpha_1\cdots\alpha_m}}^\rho:=&\sum_{\delta\in I}{\mathbf{B}_{\alpha_1\cdots\alpha_{m-1}}}^\delta{\mathbf{B}_{\delta\alpha_m}}^\rho+\hbar\cdot\partial_{\alpha_m}({\mathbf{B}_{\alpha_1\cdots\alpha_{m-1}}}^\rho),
		\end{aligned}\right.
	\]
	\[
		\left\{\begin{aligned}
			\mathbf{C}_{\alpha_1\alpha_2}:=&\Lambda_{\alpha_1\alpha_2},\\
			\mathbf{C}_{\alpha_1\alpha_2\alpha_3}:=&\sum_{\delta\in I}{\mathbf{B}_{\alpha_1\alpha_2}}^\delta\mathbf{C}_{\delta\alpha_3}+\hbar\downtriangle_{\alpha_3}^{\frac{S+\Gamma}{\hbar}}\mathbf{C}_{\alpha_1\alpha_2},\\
			&\vdots\\
			\mathbf{C}_{\alpha_1\cdots\alpha_m}:=&\sum_{\delta\in I}{\mathbf{B}_{\alpha_1\cdots\alpha_{m-1}}}^\delta\mathbf{C}_{\delta\alpha_m}+\hbar\downtriangle_{\alpha_m}^{\frac{S+\Gamma}{\hbar}}\mathbf{C}_{\alpha_1\cdots\alpha_{m-1}}.
		\end{aligned}\right.
	\]
\end{definition}

\begin{theorem}\label{Thm_UBD}
	The equations \eqref{F-QM11} hold for all $\alpha,\beta\in I$ if and only if the following inductive equations hold for all $m\geq2$ and all $\underline{\alpha}\in I^m$:
	\begin{equation}\label{Ind_QM}
		\mathbf{U}_{\underline{\alpha}}=\sum_{\rho\in I}{\mathbf{B}_{\underline{\alpha}}}^\rho\Gamma_\rho+(\delta_{S+\Gamma}+\hbar\Delta)(\mathbf{C}_{\underline{\alpha}}).
	\end{equation}
\end{theorem}

\begin{proof}
	It is trivial that the above inductive equations imply the equation \eqref{F-QM11} because the $m=2$ case of \eqref{Ind_QM} is same as \eqref{F-QM11}.\par
	Suppose that the equation \eqref{F-QM11} holds for all $\alpha, \beta\in I$. Therefore, the inductive equations hold for all $\underline{\alpha}\in I^2$. Assume that the inductive equations hold for all $\underline{\alpha}\in I^\ell$, i.e.,
	\begin{equation}\label{ind_ell}
		\mathbf{U}_{\alpha_1\cdots\alpha_\ell}=\sum_{\rho\in I}{\mathbf{B}_{\alpha_1\cdots\alpha_\ell}}^\rho\Gamma_\rho+(\delta_{S+\Gamma}+\hbar\Delta)(\mathbf{C}_{\alpha_1\cdots\alpha_\ell}),
	\end{equation}
	for all $(\alpha_1,\dots,\alpha_\ell)\in I^\ell$. Take $\hbar\downtriangle_{\alpha_{\ell+1}}^{\frac{S+\Gamma}{\hbar}}$ on the equation \eqref{ind_ell}. Then we have the following equations:
	\begin{align*}
		\hbar\downtriangle_{\alpha_{\ell+1}}^{\frac{S+\Gamma}{\hbar}}\mathbf{U}_{\alpha_1\cdots\alpha_\ell}=&\sum_{\rho\in I}\hbar\downtriangle_{\alpha_{\ell+1}}^{\frac{S+\Gamma}{\hbar}}({\mathbf{B}_{\alpha_1\cdots\alpha_\ell}}^\rho\Gamma_\rho)+(\delta_{S+\Gamma}+\hbar\Delta)(\hbar\downtriangle_{\alpha_{\ell+1}}^{\frac{S+\Gamma}{\hbar}}\mathbf{C}_{\alpha_1\cdots\alpha_\ell}),\\
		\mathbf{U}_{\alpha_1\cdots\alpha_{\ell+1}}=&\sum_{\rho\in I}\Big[\hbar\partial_{\alpha_{\ell+1}}{\mathbf{B}_{\alpha_1\cdots\alpha_\ell}}^\rho\Gamma_\rho+{\mathbf{B}_{\alpha_1\cdots\alpha_\ell}}^\rho\hbar\downtriangle_{\alpha_{\ell+1}}^{\frac{S+\Gamma}{\hbar}}\hbar\downtriangle_{\rho}^{\frac{S+\Gamma}{\hbar}}1\Big]\\
		&\hspace{5em}+(\delta_{S+\Gamma}+\hbar\Delta)(\hbar\downtriangle_{\alpha_{\ell+1}}^{\frac{S+\Gamma}{\hbar}}\mathbf{C}_{\alpha_1\cdots\alpha_\ell})\\
		=\sum_{\rho\in I}\bigg[\hbar&\partial_{\alpha_{\ell+1}}{\mathbf{B}_{\alpha_1\cdots\alpha_\ell}}^\rho\Gamma_\rho+{\mathbf{B}_{\alpha_1\cdots\alpha_\ell}}^\rho\Big[\sum_{\delta\in I}A_{\alpha_{\ell+1}\rho}{}^\delta\Gamma_\delta+(\delta_{S+\Gamma}+\hbar\Delta)(\Lambda_{\alpha_{\ell+1}\rho})\Big]\bigg]\\
		&\hspace{5em}+(\delta_{S+\Gamma}+\hbar\Delta)(\hbar\downtriangle_{\alpha_{\ell+1}}^{\frac{S+\Gamma}{\hbar}}\mathbf{C}_{\alpha_1\cdots\alpha_\ell})\\
		=\sum_{\rho\in I}\bigg[\hbar&\partial_{\alpha_{\ell+1}}{\mathbf{B}_{\alpha_1\cdots\alpha_\ell}}^\rho+\sum_{\delta\in I}\mathbf{B}_{\alpha_1\cdots\alpha_\ell}{}^\delta A_{\alpha_{\ell+1}\delta}{}^\rho\bigg]\Gamma_\rho\\
		&\hspace{5em}+(\delta_{S+\Gamma}+\hbar\Delta)(\sum_{\delta\in I}\mathbf{B}_{\alpha_1\cdots\alpha_\ell}{}^\delta\Lambda_{\alpha_{\ell+1}\delta}+\hbar\downtriangle_{\alpha_{\ell+1}}^{\frac{S+\Gamma}{\hbar}}\mathbf{C}_{\alpha_1\cdots\alpha_\ell}).
	\end{align*}
	Recall that $\hbar\downtriangle_{\alpha}^{\frac{S+\Gamma}{\hbar}}$ and $\delta_{S+\Gamma}+\hbar\Delta$ commute. The last equation implies
	\[
		\mathbf{U}_{\alpha_1\cdots\alpha_{\ell+1}}=\sum_{\rho\in I}{\mathbf{B}_{\alpha_1\cdots\alpha_{\ell+1}}}^\rho\Gamma_\rho+(\delta_{S+\Gamma}+\hbar\Delta)(\mathbf{C}_{\alpha_1\cdots\alpha_{\ell+1}})
	\]
	by the definitions of $\mathbf{U},\mathbf{B}$, and $\mathbf{C}$. By mathematical induction, the result follows.
\end{proof}

By a direct computation, we get the following lemma:
\begin{lemma}
	For $\underline{\alpha}\in I^m$, we have the following formula:
	\[
		\mathbf{U}_{\underline{\alpha}}=\sum_{i=0}^{m-1}{}^iU_{\underline{\alpha}}\hbar^i,\quad\textrm{where}\quad{}^iU_{\underline{\alpha}}=\sum_{\substack{\underline{\alpha}_1\sqcup\cdots\sqcup \underline{\alpha}_{m-i}=\underline{\alpha}\\\underline{\alpha}_j\neq\emptyset}}\frac{1}{(m-i)!}\Gamma_{\underline{\alpha}_1}\cdots\Gamma_{\underline{\alpha}_{m-i}},
	\]
	where $\underline{\alpha}_1\sqcup\cdots\sqcup \underline{\alpha}_{m-i}=\underline{\alpha}$ means $\underline{\alpha}_1\cup\cdots\cup \underline{\alpha}_{m-i}=\underline{\alpha}$, and $\underline{\alpha}_k\cap \underline{\alpha}_\ell=\emptyset$ for $\ell\neq k$.
\end{lemma}
For example,
\begin{align*}
	\mathbf{U}_{\alpha\beta\gamma}=&\Gamma_\alpha\Gamma_\beta\Gamma_\gamma+\hbar(\Gamma_\alpha\Gamma_{\beta\gamma}+\Gamma_{\beta}\Gamma_{\alpha\gamma}+\Gamma_\gamma\Gamma_{\alpha\beta})+\hbar^2\Gamma_{\alpha\beta\gamma},\\
	\mathbf{U}_{\alpha\beta\gamma\delta}=&\Gamma_\alpha\Gamma_\beta\Gamma_\gamma\Gamma_\delta\\
	&+\hbar(\Gamma_{\alpha\beta}\Gamma_{\gamma}\Gamma_{\delta}+\Gamma_{\alpha\gamma}\Gamma_{\beta}\Gamma_{\delta}+\Gamma_{\alpha\delta}\Gamma_{\beta}\Gamma_{\gamma}+\Gamma_{\beta\gamma}\Gamma_{\alpha}\Gamma_{\delta}+\Gamma_{\beta\delta}\Gamma_{\alpha}\Gamma_{\gamma}+\Gamma_{\gamma\delta}\Gamma_{\alpha}\Gamma_{\beta})\\
	&+\hbar^2(\Gamma_{\alpha\beta}\Gamma_{\gamma\delta}+\Gamma_{\alpha\gamma}\Gamma_{\beta\delta}+\Gamma_{\alpha\delta}\Gamma_{\beta\gamma}+\Gamma_\alpha\Gamma_{\beta\gamma\delta}+\Gamma_{\beta}\Gamma_{\alpha\gamma\delta}+\Gamma_\gamma\Gamma_{\alpha\beta\delta}+\Gamma_{\delta}\Gamma_{\alpha\beta\gamma})\\
	&+\hbar^3\Gamma_{\alpha\beta\gamma\delta}.
\end{align*}
Therefore, we get the following result:
\begin{equation}\label{rel_U-u}
	{}^iU_{\underline{\alpha}}|_{\underline{t}=0}=u^{(i)}_{\underline{\alpha}}=\sum_{\substack{\underline{\alpha}_1\sqcup\cdots\sqcup \underline{\alpha}_{m-i}=\underline{\alpha}\\\underline{\alpha}_j\neq\emptyset}}\frac{1}{(m-i)!}u_{\underline{\alpha}_1}\cdots u_{\underline{\alpha}_{m-i}},
\end{equation}
where $u_{\underline{\alpha}}^{(i)}$ appeared in Definition \ref{u_alpha}.\par
By the definitions of $\mathbf{B}$ and $\mathbf{C}$, the $\hbar$-degree of $\mathbf{B}_{\underline{\alpha}}{}^\rho$ and $\mathbf{C}_{\underline{\alpha}}$ is $m-2$ when $\underline{\alpha}\in I^m$, i.e., we can write
\[
	\mathbf{B}_{\underline{\alpha}}{}^\rho=\sum_{i=0}^{m-2}{}^iB_{\underline{\alpha}}{}^\rho\hbar^i,\quad\mathbf{C}_{\underline{\alpha}}=\sum_{i=0}^{m-2}{}^iC_{\underline{\alpha}}\hbar^i,
\]
for some ${}^iB_{\underline{\alpha}}{}^\rho\in\mathbb{C}[\![\underline{t}]\!]$, and ${}^iC_{\underline{\alpha}}\in\cB^{-1}[\![\underline{t}]\!]$. Therefore, by comparing the $\hbar$-power terms of \eqref{Ind_QM}, we get the following sequence of equations:
\begin{equation}\label{IND_U}
	\begin{aligned}
		{}^0U_{\underline{\alpha}}&=\sum_{\rho\in I}{}^0B_{\underline{\alpha}}{}^\rho\Gamma_\rho+\delta_{S+\Gamma}({}^0C_{\underline{\alpha}}),\\
		{}^1U_{\underline{\alpha}}&=\sum_{\rho\in I}{}^1B_{\underline{\alpha}}{}^\rho\Gamma_\rho+\delta_{S+\Gamma}({}^1C_{\underline{\alpha}})+\Delta({}^0C_{\underline{\alpha}}),\\
		&\quad\vdots\\
		{}^{m-2}U_{\underline{\alpha}}&=\sum_{\rho\in I}{}^{m-2}B_{\underline{\alpha}}{}^\rho\Gamma_\rho+\delta_{S+\Gamma}({}^{m-2}C_{\underline{\alpha}})+\Delta({}^{m-3}C_{\underline{\alpha}}),\\
		{}^{m-1}U_{\underline{\alpha}}&=\Delta({}^{m-2}C_{\underline{\alpha}}).
	\end{aligned}
\end{equation}

\begin{lemma}\label{Lemm_B}
	The constant term of $\underline{t}^{\ud \a}$-expansion of ${}^{m-2}B_{\underline{\alpha}}{}^\rho$ is $a_{\underline{\alpha}}{}^\rho$ which appeared in \eqref{Exp_notation}, i.e.,
	\[
		{}^{m-2}B_{\underline{\alpha}}{}^\rho|_{\underline{t}=0}=a_{\underline{\alpha}}{}^\rho.
	\]
\end{lemma}

\begin{proof}
	It is enough to show that
	\begin{equation}\label{Eqn_B}
		{}^{m-2}B_{\underline{\alpha}}{}^\rho=\partial_{\alpha_m}\partial_{\alpha_{m-1}}\cdots\partial_{\alpha_3}A_{\alpha_1\alpha_2}{}^\rho,
	\end{equation}
	where $\underline{\alpha}=\alpha_1\alpha_2\cdots\alpha_m\in I^m$ for $m\geq 2$. Use the mathematical induction. When $m=2$, $\mathbf{B}_{\alpha_1\alpha_2}{}^\rho=A_{\alpha_1\alpha_2}{}^\rho$ by Definition \ref{Def_UBC}, which implies that ${}^0B_{\alpha_1\alpha_2}{}^\rho=A_{\alpha_1\alpha_2}{}^\rho$.\par
Suppose that equation \eqref{Eqn_B} is true for $|\underline{\alpha}|=\ell$. By Definition \ref{Def_UBC},
	\[
		{\mathbf{B}_{\alpha_1\cdots\alpha_{\ell+1}}}^\rho=\sum_{\delta\in I}{\mathbf{B}_{\alpha_1\cdots\alpha_{\ell}}}^\delta{\mathbf{B}_{\delta\alpha_{\ell+1}}}^\rho+\hbar\cdot\partial_{\alpha_{\ell+1}}({\mathbf{B}_{\alpha_1\cdots\alpha_{\ell}}}^\rho).
	\]
	Therefore, the $\hbar^{\ell-1}$-term of $\mathbf{B}_{\alpha_1\cdots\alpha_{\ell+1}}{}^\rho$ is $\partial_{\alpha_{\ell+1}}({}^{\ell-2}B_{\alpha_1\cdots\alpha_\ell}{}^\rho)=\partial_{\alpha_{\ell+1}}\cdots\partial_{\alpha_3}A_{\alpha_1\alpha_2}{}^\rho$.
\end{proof}
According to Lemma \ref{Lemm_B} and equation \eqref{rel_U-u}, we get equations \eqref{ind_u} by evaluating $\underline{t}=0$ on equations \eqref{IND_U}. This implies that the algorithm in Section \ref{sec3.4} works to give a solution to \eqref{F-QM11}.

\section{Comparison with Li-Li-Saito's algorithm}

It is easy to see that equations \eqref{WPF} and \eqref{F-QM11} imply the following proposition.
\begin{proposition}
If the flat $F$-manifold $(J_S, A_{\a\b}{}^{\g}(\ud t))$ constructed in Section \ref{sec3.4} can be extended to a formal Frobenius manifold $(J_S, A_{\a\b}{}^{\g}(\ud t), g_{\a\b})$, then $(J_S, A_{\a\b}{}^{\g}(\ud t), g_{\a\b})$ is associated with (as in \cite{ST}) a primitive form $[1] \in J_S$.
\end{proposition}

We consider a simple elliptic singularity. Let 
\begin{equation}
    S(\ud x) = \frac{1}{3} (x_1^3+x_2^3+x_3^3).
\end{equation}
Note that $S(\ud x)$ is a weight homogeneous polynomial, called a simple elliptic singularity of type $E_6^{(1,1)}$. By \cite[Proposition 6.3]{LLS}, a primitive form can not be [1] and, thus the formal flat $F$-manifold obtained from the formal Frobenius manifold (associated with the primitive form $[1]$) constructed in \cite[Proposition 6.3]{LLS}, by forgetting the metric, is different from the formal $F$-manifold constructed in this article.
This indicates that our algorithm is different from that of Li-Li-Saito (\cite{LLS} and \cite{LLSS}). We find that it seems difficult to extend our algorithm to construct formal Frobenius manifolds.



\begin{thebibliography}{99}

\bibitem{BK}
Barannikov, S.; Kontsevich, M.:\newblock{\it Frobenius manifolds and formality of Lie algebras of polyvector fields}, \newblock Internat. Math. Res. Notices {\bf 4} (1998) 201--215.


\bibitem{CS99} Cox, David A.; Katz, Sheldon:{\it Mirror symmetry and algebraic geometry.} Mathematical Surveys and Monographs, 68. American Mathematical Society, Providence, RI, 1999. xxii+469 pp.

\bibitem{Dim95} Dimca, A.:\newblock{\it Residues and cohomology of complete intersections}, \newblock Duke Math. J., 78 (1995) no. 1, 89--100.

\bibitem{Dubrovin} Dubrovin,  B. A.:\newblock{\it Geometry of $2$D topological field theories}, \newblock In  Integrable systems and quantum groups (Montecatini Terme, 1993), Lecture Notes in Math.\ vol.\ 1620, 120--348. Springer-Verlag, Berlin (1996).


\bibitem{Gr69} Griffiths, Phillip A.: \newblock{\it  On the periods of certain rational integrals. I, II}, \newblock Ann. of Math. (2) 90 (1969), 460--495; ibid. (2) 90 (1969), 496--541. 

\bibitem{Her02} Hertling, Claus: {\it Frobenius manifolds and moduli spaces for singularities.} Cambridge Tracts in Mathematics, 151. Cambridge University Press, Cambridge, 2002. x+270 pp. 

\bibitem{HM} Hertling, C.; Manin, Y.:\newblock{\it Weak Frobenius Manifolds}, \newblock Internat. Math. Res. Notices {\bf 6} (1999) 277--286.



\bibitem{Ko91}
K. Konno:  \emph{On the variational Torelli problem for complete intersections}, Comp. Math., 78 (1991), 271--296.

\bibitem{LLS}
Li, Changzheng; Li, Si; Saito, Kyoji: 
{\it Primitive forms via polyvector fields,} available at https://arxiv.org/abs/1311.1659

\bibitem{LLSS}
Li, Changzheng; Li, Si; Saito, Kyoji; Shen, Yefeng: 
{\it Mirror symmetry for exceptional unimodular singularities,} J. Eur. Math. Soc. 9, 1189--1229.

\bibitem{Ma05} Manin, Y.: {\it $F$-manifolds with flat structure and Dubrovin's duality.} Adv. Math. 198 (2005), no. 1, 5--26.






\bibitem{Saito}
Saito, K.:
\newblock{\it Primitive forms for an universal unfolding of a functions with isolated
critical point},
\newblock Journ.\ Fac.\ Sci.\ Univ.\ Tokyo, Sect. IA Math.  {\bf 28} no.3 (1981) 777--792. 



\bibitem{MSaito}
Saito, M.:
\newblock{\it On the structure of Brieskorn lattice},
Ann. Inst. Fourier (Grenoble) {\bf 39} (1989), no.1, 27--72.

\bibitem{ST}
Saito, K.; Takahashi, A.:
\newblock{\it From primitive forms to Frobenius manifolds,}
From Hodge theory to integrability and TQFT $tt^*$-geometry, 31--48, Proc. Sympos. Pure Math., 78, Amer. Math. Soc., Providence, RI, 2008.



\end{thebibliography}
\end{document}